\documentclass[smallcondensed,envcountsame,envcountsect]{svjour3}

\usepackage{amsmath,amssymb}
\usepackage{mathptmx,enumerate}
\usepackage{algpseudocode,algorithm}

\smartqed

\newcommand{\I}{\mathcal I}
\newcommand{\rmv}[1]{}
\newcommand{\Tr}{\mathrm{Tr}}
\newcommand{\fqn}{\mathbb{F}_{q^n}}
\newcommand{\F}{\mathbb{F}}
\newcommand{\fq}{\mathbb{F}_q}
\newcommand{\ord}{\mathrm{ord}}
\newcommand{\lcm}{\mathrm{lcm}}

\title{Variations of the Primitive Normal Basis Theorem}
\date{\today}
\author{Giorgos Kapetanakis \and Lucas Reis}
\institute{
	Giorgos Kapetanakis
	\at Faculty of Engineering and Natural Sciences, Sabanc\i{} \"{U}niversitesi. Ortha Mahalle, Tuzla 34956, \.{I}stanbul, Turkey \\\email{gnkapet@gmail.com} \\
	Lucas Reis
	\at School of Mathematics and Statistics, Carleton University, 1125 Colonel By Drive, Ottawa ON (Canada), K1S 5B6
	\at Permanent address: 
Departamento de Matem\'{a}tica,
Universidade Federal de Minas Gerais,
UFMG,
Belo Horizonte MG (Brazil),
 30123-970.\\\email{lucasreismat@gmail.com}
	}
	
\begin{document}
\maketitle
\begin{abstract}
The celebrated Primitive Normal Basis Theorem states that for any $n\ge 2$ and any finite field $\F_q$, there exists an element $\alpha\in \F_{q^n}$ that is simultaneously primitive and normal over $\F_q$. In this paper, we prove some variations of this result, completing the proof of a conjecture proposed by Anderson and Mullen (2014). Our results also imply the existence of elements of $\fqn$ with multiplicative order $(q^n-1)/2$ and prescribed trace over $\F_q$.
\keywords{primitive elements \and normal bases \and $k$-normal elements \and high-order elements}
\subclass{12E20 \and 11T30}
\end{abstract}

\maketitle

\section{Introduction}
Let $\F_q$ be the finite field with $q$ elements, where $q$ is a power of a prime $p$ and let $n\ge 1$ be a positive integer. We recall that the multiplicative group $\F_{q^n}^*$ is cyclic and any generator of this group is called {\it primitive}. Primitive elements have numerous applications in areas like cryptography; perhaps the most notable such example is the widely used Diffie-Hellman key exchange \cite{dh}. Also, $\F_{q^n}$ can be regarded as a $\F_q$-vector space of dimension $n$: an element $\alpha\in \F_{q^n}$ is {\it normal} over $\F_q$ if $\mathcal B=\{\alpha, \ldots, \alpha^{q^{n-1}}\}$ is a basis of $\F_{q^n}$. In this case, $\mathcal B$ is called a {\it normal basis}. For many practical applications, such as cryptography and computer algebra systems, it is more efficient to work with normal bases. For a comprehensive coverage on normal bases and their importance, both in theory and applications, we refer to \cite{gao93} and the references therein. 

Sometimes it is also desired that such normal bases are composed by primitive elements. The Primitive Normal Basis Theorem states that there exists normal basis composed by primitive elements in any finite field extension. The proof of this result was first presented by Lenstra and Schoof \cite{lenstra}, and a proof without the use of a computer was given by Cohen and Huczynska \cite{cohen}. Recently, Hachenberger \cite{hachenberger15}, using geometric tools, established sharp estimates for the number of such bases.

A variation of normal elements was recently introduced by Huczynska et al.~\cite{HMPT}, yielding $k$-normal elements. There are many equivalent definitions for such elements and here we pick the most natural.

\begin{definition}
For $\alpha\in \F_{q^n}$, consider the set $S_{\alpha}=\{\alpha, \alpha^q, \ldots, \alpha^{q^{n-1}}\}$ comprising the conjugates of $\alpha$ by the action of the Galois group of $\F_{q^n}$ over $\F_q$. The element $\alpha$ is \emph{$k$-normal over $\F_q$} if the $\F_q$-vector space $V_{\alpha}$ generated by $S_{\alpha}$ has dimension $n-k$, i.e., $V_{\alpha}\subseteq \F_{q^n}$ has co-dimension $k$. 
\end{definition}

Following this definition, $0$-normal elements are just the usual normal elements and $0\in \F_{q^n}$ is the only $n$-normal element. Also, we note that the definition of $k$-normal depends strongly on the base field that we are working and, unless otherwise stated, $\alpha\in \F_{q^n}$ is $k$-normal if it is $k$-normal over $\F_q$. 

We recall that the multiplicative group $\F_{q^n}^*$ has $q^n-1$ elements and, for $\alpha\in \F_{q^n}^*$, the multiplicative order of $\alpha\in \F_{q^n}^*$ is the least positive integer $d$ such that $\alpha^d=1$. We write $d=\ord(\alpha)$. Since $\alpha^{q^n-1}=1$, $d$ is always a divisor of $q^n-1$. For instance, the primitive elements are the ones of order $q^n-1$. We know that, for each divisor $e$ of $q^n-1$, there exist $\varphi(e)$ elements in $\F_{q^n}$ with order $e$, where $\varphi$ is the {\it Euler Phi} function. We introduce a variation of primitive elements in finite fields.

\begin{definition}
For $\alpha\in \F_{q^n}^*$ and $r$ a divisor of $q^n-1$, $\alpha$ is \emph{$r$-primitive} if $\ord(\alpha)=\frac{q^n-1}{r}$.
\end{definition}
From definition, the $1$-primitive elements correspond to the primitive elements in the usual sense. Motivated by the Primitive Normal Basis Theorem, in 2014, Anderson and Mullen propose the following problem (see \cite{Fq12}, Conjecture~3).

\begin{conjecture}[Anderson-Mullen]\label{MA} Suppose that $p\ge 5$ is a prime and $n\ge 3$. Then, for $a=1,2$ and $k=0, 1$, there exists some $k$-normal element $\alpha\in \F_{p^n}$ with multiplicative order $(p^n-1)/a$. \end{conjecture}

In other words, if $p\ge 5$ and $n\ge 3$, there exists an element $\alpha\in \F_{p^n}$ that is simultaneously $a$-primitive and $k$-normal, for $a=1, 2$ and $k=0, 1$. We notice that the case $(a, k)=(1, 0)$ is the Primitive Normal Basis Theorem, which holds for arbitrary finite fields. Also, the case $(a, k)=(1, 1)$ was recently proved for arbitrary $q$ and $n\ge 3$  (see \cite{RT}), yielding the Primitive $1$-normal Basis Theorem.

In this paper, we complete the proof of Conjecture~\ref{MA} above, adding the case $a=2$. In fact, we prove a stronger version of this conjecture.

\begin{theorem}\label{thm:main_result}
Let $q$ be a power of a prime $p$ and let $n$ be a positive integer. 
\begin{enumerate}
\item 
If $p\geq 3$ and $n\geq 3$, there exists an element $\alpha\in \F_{q^n}$ that is simultaneously $2$-primitive and normal over $\F_{q}$, with the sole exception $(q,n)=(3,4)$. Furthermore, with the exception of the case $q=3$, there is no $2$-primitive normal element of $\F_{q^2}$ over $\F_{q}$.
\item
If $p\geq 5$ and $n\geq 2$, there exists an element $\alpha\in \F_{q^n}$ that is simultaneously $2$-primitive and $1$-normal over $\F_{q}$.
\end{enumerate}
\end{theorem}

Moreover, in Theorem~\ref{thm:prescribed_trace}, given $q$ a power of an odd prime $p$, we prove the existence of $2$-primitive elements of $\fqn$ with prescribed trace over $\F_q$ provided that $p \nmid n$ and $p,n\geq 3$ or that $p\mid n$, $n\geq 3$ and $p\geq 5$.

We make a brief comment on the techniques used in this paper. We use standard characteristic functions to describe a special class of $2$-primitive, $k$-normal elements (for $k=0, 1$) over finite fields: these characteristic functions can be described via \emph{character sums}. This characterization provides \emph{sieve inequalities} for the existence of such elements and then we study these inequalities in both theoretical and computational aspects.

\section{Preliminaries}
In this section, we provide a background material for $k$-normal elements, as well as some particular artihmetic functions and their polynomial version.

\begin{definition}
\begin{enumerate}[(a)]
\item Let $f$ be a monic polynomial with coefficients in $\F_q$. The \emph{Euler Phi Function} for polynomials over $\F_q$ is given by $$\Phi(f)=\left |\left(\frac{\F_q[x]}{\langle f\rangle}\right)^{*}\right |,$$ where $\langle f\rangle$ is the ideal generated by $f(x)$ in $\F_q[x]$. 
\item If $t$ is a positive integer (or a monic polynomial over $\F_q$), $W(t)$ denotes the number of square-free (monic) divisors of $t$.
\item If $f$ is a monic polynomial with coefficients in $\F_q$, the \emph{Polynomial M\"obius Function} $\mu_q$ is given by $\mu_q(f)=0$ is $f$ is not square-free and $\mu_q(f)=(-1)^r$ if $f$ writes as a product of $r$ distinct irreducible factors over $\F_q$.  
\end{enumerate}
\end{definition}

\subsection{Additive order of elements and $k$-normals}
If $f\in \F_q[x]$, $f=\sum_{i=0}^{s}a_ix^i$, we define $L_f(x)=\sum_{i=0}^{s}a_ix^{q^i}$ as the $q$-associate of $f$. Also, for $\alpha\in \F_{q^n}$, set $f\circ \alpha=L_{f}(\alpha)=\sum_{i=0}^sa_i\alpha^{q^i}$.
As follows, the $q$-associates have a good behavior through basic operations of polynomials.
\begin{lemma}[\cite{LN}, Theorem 3.62]
Let $f, g\in \F_q[x]$. The following hold:
\begin{enumerate}[(i)]
\item $L_f(L_g(x))=L_{fg}(x)$,
\item $L_f(x)+L_g(x)=L_{f+g}(x)$.
\end{enumerate}
\end{lemma}

Notice that, for any element $\alpha$ in some extension of $\F_{q}$,  $(x^n-1)\circ \alpha=\alpha^{q^n}-\alpha=0$ if and only if $\alpha\in \F_{q^n}$. If we set  $\I_{\alpha}=\{g(x)\in \F_{q}[x]\,|\, g(x)\circ \alpha=0\}$, the previous lemma shows that $\I_{\alpha}$ is an ideal of $\F_{q}[x]$. In particular, for $\alpha\in \F_{q^n}$, $x^n-1\in \I_{\alpha}$ and so $\I_{\alpha}$ is a non zero ideal, hence is generated by a polynomial, say $m_{\alpha}(x)$. Notice that, if we require $m_{\alpha}(x)$ to be monic, such $m_{\alpha}(x)$ is unquely determined by $\alpha$. We define $m_{\alpha}(x)$ as the {\it $\F_q$-order} of $\alpha$. This concept works as an ``additive'' analogue of multiplicative order over finite fields. 

Clearly, for $\alpha\in \F_{q^n}$, $m_{\alpha}(x)$ divides $x^n-1$ and then its degree is at most $n$. For instance, if $\deg (m_{\alpha}(x))=0$, then $m_{\alpha}(x)=1$ and $\alpha=0$. The following result shows a connection between $k$-normal elements and their $\F_q$-order.

\begin{proposition}[\cite{HMPT}, Theorem 3.2]
Let $\alpha \in \F_{q^n}$. Then $\alpha$ is $k$-normal if and only if $m_{\alpha}(x)$ has degree $n-k$. 
\end{proposition}
For instance, normal elements $\alpha\in \F_{q^n}$ are the ones such that $m_{\alpha}(x)=x^n-1$. The previous proposition shows that the existence of $k$-normal elements depends on the existence of a polynomial of degree $n-k$ dividing $x^n-1$ over $\F_q$. In fact, according to Theorem 3.5 of \cite{HMPT}, if $g\in \F_{q}[x]$ is a divisor of $x^n-1$ degree $k$, there exist $\Phi(g)$ elements $\alpha \in \F_q$ with $m_{\alpha}(x)=g$.



Since $x-1$ divides $x^n-1$ for any $n\ge 1$, we see that $1$-normal elements exist in any extension of $\F_q$. Clearly, this also implies the existence of $(n-1)$-normal elements. Recall that $0\in \F_{q^n}$ is $n$-normal and that $0$-normal elements always exist. These are the only values of $k$ for which the existence of $k$-normal elements is guaranteed
in any finite field extension. In fact, suppose that $n$ is a prime and $q$ is primitive $\pmod n$: the polynomial $x^n-1$ factors as $(x-1) (x^{n-1}+\cdots+x+1)$ over $\F_q$. In particular, from the previous proposition, there are no $k$-normal elements in $\F_{q^n}$ for any $1<k<n-1$. Of course, in this paper, we are only interested in $0$ and $1$-normal elements.

\subsection{Application of the method of Lenstra and Schoof}
Here we present the traditional method of Lenstra and Schoof \cite{lenstra} in the characterization of elements in $\F_{q^n}$ with special properties like normal, primitive and of a given prescribed trace over some subfield $\F_{q^m}$ of $\F_{q^n}$. We start with the concept of {\it freeness}.

\begin{definition}
\begin{enumerate}
\item If $m$ divides $q^n-1$, an element $\alpha \in \F_{q^n}^*$ is \emph{$m$-free} if $\alpha = \beta^d$ for any divisor $d$ of $m$ implies $d=1$. 
\item If $m\in\F_q[x]$ divides $x^n-1$, an element $\alpha\in \F_{q^n}$ is \emph{$m$-free} if $\alpha = h \circ \beta$ for any divisor $h$ of $m$ implies $h=1$. 
\end{enumerate}
\end{definition}
From definition, the primitive elements correspond to the $(q^n-1)$-free elements. Also, the $(x^n-1)$-free elements are just the normal elements. Using the concept of freeness, we characterize special classes of $1$-normal elements, via trace functions. First, we have the following lemma.

\begin{lemma}[\cite{HMPT}, Theorem 5.4]\label{freeness}
Let $\alpha$ be any element in $\F_{q^n}$ and let $f(x)=m_{\alpha}(x)$ be its $\F_q$-order. For any divisor $g(x)$ of $x^n-1$, the following are equivalent:
\begin{enumerate}[(a)]
\item $\alpha$ is $g(x)-$free,
\item $g(x)$ and $\frac{x^n-1}{f(x)}$ are coprime.
\end{enumerate}
\end{lemma}

As follows, we show that we have a characterization of elements $\alpha\in \F_{q^n}$ for which $m_{\alpha}(x)=\frac{x^n-1}{x-1}$.

\begin{proposition}\label{LR}
Let $q$ be a power of  a prime $p$ and $n=p^k u$, where $k\ge 0$ and $\gcd(u, p)=1$. Write $T(x)=\frac{x^u-1}{x-1}$. Then $\alpha\in \F_{q^n}$ is such that $m_{\alpha}(x)=\frac{x^n-1}{x-1}$ if and only if $\alpha$ is $T(x)$-free and $\beta=\Tr_{q^n/q^{p^k}}(\alpha)$ is such that $m_{\beta}(x)=\frac{x^{p^k}-1}{x-1}$. 
\end{proposition}

\begin{proof}
If $m_{\alpha}(x)=\frac{x^n-1}{x-1}$, since $\beta=\Tr_{q^n/q^{p^k}}(\alpha)=\frac{x^{n}-1}{x^{p^k}-1}\circ \alpha$, it follows that $m_{\beta}(x)=\frac{x^{p^t}-1}{x-1}$. Conversely, suppose that $m_{\beta}(x)=\frac{x^{p^t}-1}{x-1}$ and $\beta=\Tr_{q^n/q^{p^k}}(\alpha)$.  Clearly $\frac{x^n-1}{x-1}\circ \alpha=\frac{x^{p^k}-1}{x-1}\circ (\frac{x^{n}-1}{x^{p^k}-1}\circ\alpha)=\frac{x^{p^k}-1}{x-1}\circ \beta=0$. Since $\alpha$ is $T(x)$-free, it follows from Lemma \ref{freeness} that $m_{\alpha}(x)=T(x)^{p^k}(x-1)^{d}$ for some $d\ge 0$ with $d\le p^k-1$. Since $m_{\beta}(x)=\frac{x^{p^k}-1}{x-1}$, from the minimality of $m_{\beta}$ we have $d=p^k-1$, i.e. $m_{\alpha}(x)=\frac{x^n-1}{x-1}$. 
\qed\end{proof}

In the case when $n$ is divisible by $p^2$, as noticed in \cite{RT}, we have an alternative characterization for such $1$-normal elements. 

\begin{proposition}\label{prop:1-trace}
Suppose that $n=p^2s$ and let $\alpha\in \F_{q^n}$ such that $\Tr_{q^n/q^{ps}}(\alpha)=\beta$. Then $m_{\alpha}=\frac{x^n-1}{x-1}$ if and only if $m_{\beta}=\frac{x^{ps}-1}{x-1}$. 
\end{proposition}

\noindent The proof of this proposition quite simple and can be found in Lemma 5.2 of \cite{RT}.

\subsection{Some characteristic functions}

The concept of {\it freeness} derives some characteristic functions for primitive and normal elements. We pick the notation of \cite{HMPT}.

\paragraph{Multiplicative component.} $\int\limits_{d|m}\eta_d$ stands for the sum $\sum_{d|m}\frac{\mu(d)}{\varphi(d)}\sum_{(d)}\eta_d$, where $\mu$ and $\varphi$ are the M\"obius and Euler functions for integers, respectively, $\eta_d$ is a typical multiplicative character of $\F_{q^n}$ of order $d$, and the sum $\sum_{(d)}\eta_d$ runs through all the multiplicative characters of order $d$.

\paragraph{Additive component.} $\chi$ is the canonical additive character on $\F_{q^n}$, i.e. $$\chi(\omega)=\lambda(\Tr_{q^n/q}(\omega)),\quad\omega\in \F_{q^n},$$ where $\lambda$ is the canonical additive character of $\F_q$ to $\F_{p}$. If $D$ is a monic divisor of $x^n-1$ over $\F_q$, a typical character $\chi_*$ of $\F_{q^n}$ of {\it $\F_q$-order} $D$ is one such that $\chi_*(D\circ^{q} \cdot)$ is the trivial additive character in $\F_{q^n}$ and $D$ is minimal (in terms of degree) with this property. Here, $\Delta_D$ denotes the set of all $\delta\in \F_{q^n}$ such that $\chi_{\delta}$ has $\F_q$-order $D$, where $\chi_{\delta}(\omega)=\chi(\delta\omega)$ for any $\omega\in \F_{q^n}$. For instance, $\Delta_{1}=\{0\}$ and $\Delta_{x-1}=\F_q^*$. Furthermore, it is well-known that
\[
|\Delta_D| = \Phi(D).
\]
Following our notation, $\int\limits_{D|T}\chi_{\delta_D}$ stands for the sum $$\sum_{D|T}\frac{\mu_q(D)}{\Phi(D)}\sum_{(\delta_D)}\chi_{\delta_D},$$ where $\mu_q$ and $\Phi$ are the M\"obius and Euler functions for polynomials over $\F_q$ respectively, $\chi_{\delta_D}$ is a typical additive character of $\F_{q^n}$ of $\F_q$-order $D$ and the sum $\sum_{(\delta_D)}\chi_{\delta_D}$ runs through all the additive characters of $\F_q$-order $D$, i.e. all $\delta_D\in\Delta_D$.

For each divisor $t$ of $q^n-1$ and each monic divisor $D$ of $x^n-1$, set $\theta(t)=\frac{\varphi(t)}{t}$ and $\Theta(D)=\frac{\Phi(D)}{q^{\deg D}}$. The sums above yield characteristic functions.

\begin{theorem}[\cite{HMPT}, Section~5.2]\label{thm:charfree}
\begin{enumerate}
\item For $w \in \fqn^*$ and $t$ be a positive divisor of $q^n-1$, 
\[\omega_t(w) = \theta(t) \int_{d|t} \eta_{d}(w) = \begin{cases} 1 & \text{if $w$ is $t$-free,} \\ 0 & \text{otherwise.} \end{cases}\]
\item For $w \in \fqn$ and $D$ be a monic divisor of $x^n-1$,
\[\Omega_D(w) = \Theta(D) \int_{E|D} \chi_{\delta_E}(w) = \begin{cases} 1 & \text{if $w$ is $D$-free,} \\ 0 & \text{otherwise.} \end{cases}\]
\end{enumerate}
\end{theorem}
\noindent For any divisor $m$ of $n$, we know that $\F_{q^m}$ is a subfield of $\F_{q^n}$. Let 
\[ T_{m,\beta}(w) = \begin{cases} 1 & \text{if $\Tr_{q^n/q^m}(w) = \beta$,} \\
                                       0 & \text{otherwise.}
				             \end{cases}
\]
We need a character sum formula for $T_{m,\beta}$. 
Let $\lambda$ and $\lambda_m$ be the canonical additive characters of $\fq$ and $\F_{q^m}$, respectively: the character $\lambda$  lifts $\lambda_m$ and $\chi$ to $\F_{q^m}$ and $\fqn$, respectively. In other words, $\lambda_m(w) = \lambda(\Tr_{q^m/q}(w))$ and $\chi(w) = \lambda(\Tr_{q^n/q}(w)) = \lambda(\Tr_{q^m/q}(\Tr_{q^n/q^m}(w))) = \lambda_m(\Tr_{q^n/q^m}(w))$, since the trace function is transitive. We observe that $T_{m,\beta}$ can be written as
\[ T_{m,\beta}(w) = \frac{1}{q^m} \sum_{d \in \F_{q^m}} \lambda_m(d (\Tr_{q^n/q^m}(w) - \beta)) = \frac{1}{q^m} \sum_{d \in \F_{q^m}} \lambda_m(d \cdot \Tr_{q^n/q^m}(w-\alpha)),\]
for any $\alpha \in \fqn$ such that $\Tr_{q^n/q^m}(\alpha) = \beta$, since 
\[ \sum_{d\in \F_{q^m}} \lambda_m(d \cdot \Tr_{q^n/q^m}(w-\alpha))= q^m\]
if and only if $\Tr_{q^n/q^m}(w)=\Tr_{q^n/q^m}(\alpha)=\beta$ and, otherwise, this sum equals $0$.
In particular, $$T_{m,\beta}(w) = \frac{1}{q^m} \sum_{d \in \F_{q^m}}\chi_d(w-\alpha)=\frac{1}{q^m} \sum_{d \in \F_{q^m}}\chi_d(w)\chi_d(\alpha)^{-1}.$$ 

More specifically, for $t=q^n-1$ and $D=x^n-1$, we obtain characteristic functions for primitive and normal elements, respectively. We write $\omega_{q^n-1}=\omega$ and $\Omega_{x^n-1}=\Omega$. As usual, we may extend the multiplicative characters to $0$ by setting $\eta_1(0)=1$, where $\eta_1$ is the trivial multiplicative character and $\eta(0)=0$ if $\eta$ is not trivial.

For a more detailed account of the above, we refer the interested reader to \cite{cohen2} and the references therein.

Also, write $\Omega^{(1)}$ the characteristic function for elements $\alpha\in \F_{q^n}$ such that $m_{\alpha}(x)=\frac{x^n-1}{x-1}$; from Proposition \ref{LR}, if $n=p^t u$ with $\gcd(u, p)=1$ and $\beta$ is any element of $\F_{q^{p^t}}$ such that $m_{\beta}(x)=\frac{x^{p^t}-1}{x-1}$, then 
$$\Omega^{(1)}=\Omega_T\cdot T_{p^t, \beta},$$ 
where $T(x)=\frac{x^u-1}{x-1}$.

We may characterize ($0$ and $1$)-normal elements with prescribed multiplicative order $(q^n-1)/2$ (i.e., $2$-primitive); the element $b\in \F_{q^n}$ has order $(q^n-1)/2$ if and only if $b=a^2$ for some primitive element $a$. 

The following result is straightforward.
\begin{proposition}\label{char1}
For $w\in \F_{q^n}$, the following hold:

\begin{enumerate}
\item $w^2$ is $2$-primitive and normal if and only if $\omega(w)\cdot \Omega(w^2)=1$.
\item $w^2$ is $2$-primitive with $m_{w^2}(x)=\frac{x^n-1}{x-1}$ if and only if $\omega(w)\cdot \Omega^{(1)}(w^2)=1$.
\end{enumerate}

\end{proposition}

The following two results provide some character sum estimates that are useful.

\begin{lemma}[\cite{LN}, Theorem 5.41]\label{Gauss1}
Let $\chi$ be an additive character of $\F_{q^n}$ and $f\in \F_{q^n}[x]$ be a monic polynomial of positive degree, not of the form $g(x)^p-g(x)+y$ for any $g\in \F_{q^n}[x]$. Suppose that $e$ is the number of distinct roots of $f$ in its splitting field over $\F_{q^n}$. For every $a\in \F_{q^n}$, 
$$\left|\sum_{c\in \F_{q^n}}\chi(af(c))\right|\le (e-1)q^{n/2}.$$
\end{lemma}



\begin{theorem}[\cite{schmidt}, Theorem~2G]\label{Gauss2}
Let $\eta$ be a multiplicative character of $\F_{q}$ of order $d\neq 1$ and $\chi$ a non-trivial additive character of $\F_{q}$. If $f,g\in\F_{q}[x]$ are such that $f$ has exactly $m$ roots and $\deg(g)=n$ with $\gcd(d,\deg(f))=\gcd(n,q)=1$, then
\[
\left| \sum_{c\in\F_q} \eta(f(c))\chi(g(c)) \right| \leq (m+n-1) q^{1/2} .
\]
\end{theorem}

\section{Sieving inequalities for $2$-primitive, $f$-free elements with prescribed trace}
For $\chi$ and $\eta$, additive and multiplicative characters of $\F_{q^n}$ respectively, $G_2(\eta, \chi)$ stands for the sum $\sum_{w\in \F_{q^n}} \eta(w)\chi(w^2)$. 

\begin{theorem}\label{thm:main_ineq}
Write $n=p^ku$, with $\gcd(u, p)=1$ and let $\F_{q^n}\supsetneq \F_{q^m}\supseteq \F_{q}$. Let $f$ be a polynomial not divisible by $x-1$ such that $f$ divides $x^u-1$. Let $N(f, m, \beta)$ denote the number of elements $w\in \F_{q^n}$ such that $w$ is primitive, $w^2$ is $f$-free over $\F_q$ and $\Tr_{q^n/q^m}(w^2)=\beta$, where $\beta\in \F_{q^m}$. suppose that $m=1$ or $m$ is a power of $p$ if $f\ne 1$. Additionally, let $N_{2}(n)$ be the number of primitive elements $w\in \F_{q^n}$ such that $w^2$ is normal (i.e., $x^n-1$-free) over $\F_q$. 
\begin{enumerate}[(i)]
 \item If $\alpha$ is any element of $\F_{q^n}$ such that $\Tr_{q^n/q^m}(\alpha)=\beta$ and $a_c=\chi_c(\alpha)^{-1}$, the following holds
\begin{align*}
\frac{N(f, m, \beta)}{\theta(q^n-1)\Theta(f)}= \frac{1}{q^m} &\left[q^n + \sum_{c\in \F_{q^m}}\right. a_c\int\limits_{d|q^n-1\atop{d\ne 1}} \int\limits_{D|f\atop {D\ne 1}}G_2(\eta_d, \chi_{\delta_D+c}) 
+ \left.\sum_{c\in \F^*_{q^m}}a_c\int\limits_{d|q^n-1\atop{d\ne 1}} G_2(\eta_d, \chi_{c})\right].
\end{align*}

In particular, we have the following inequality:
\begin{equation}\label{eqn:main1}\frac{N(f, m, \beta)}{\theta(q^n-1)\Theta(f)}> q^{n-m}-2q^{n/2}W(q^n-1)W(f).\end{equation}

\item We have that $\frac{N_2(n)}{\theta(q^n-1)}=q^n+\int\limits_{d|q^n-1\atop d\ne 1}\int\limits_{D|x^n-1\atop D\ne 1}G_2(\eta_d, \chi_{\delta_D})$ and, in particular, $$\frac{N_2(n)}{\theta(q^n-1)}>q^n-2q^{n/2}W(x^n-1)W(q^n-1).$$

\end{enumerate}

\end{theorem}

\begin{proof}
We just prove item (i) since item (ii) is easier and follows by similar ideas.
Combining the characteristic functions for primitivity, $f$-free and prescribed trace, we obtain the following equality:

$$N(f, m, \beta)=\sum_{w\in \F_{q^n}}\Omega_f(w^2)T_{m, \beta}(w^2)\omega(w),$$
hence

$$\frac{N(f, m, \beta)}{\theta(q^n-1)\Theta(f)}=\frac{1}{q^m}\sum_{c\in \F_{q^m}}a_c \int\limits_{d|q^n-1}\displaystyle\int\limits_{D|f}G_2(\eta_d, \chi_{\delta_D+c}).$$

We first simplify the sum above by eliminating the trivial sums $G_2(\eta_d, \chi_{\delta_D+c})$. We observe that $\delta_D+c\ne 0$, unless $D=1$ and $c=0$: for this, we note that for $f=1$, we only have $D=1$ and so $\delta_D=0$, hence $\delta_D+c\ne 0$ unless $c=0$. If $f\ne 1$, from hypothesis, $m=1$ or $m$ is a power of $p$. Since $f$ is not divisible by $x-1$ and $x^m-1=(x-1)^m$, $\delta_D$ is never an element of $\F_{q^m}$, unless $\delta_D=0$, i.e., $D=1$. In particular, since $c\in \F_{q^m}$, it follows that $\delta_D+c\ne 0$ unless $D=1$ and $c=0$.

This shows that $G_2(\eta_1, \chi_{\delta_D+c})=q^n$ if $D=1$ and $c=0$ and, otherwise the orthogonality relations and Lemma~\ref{Gauss1} imply, $G_2(\eta_1, \chi_{\delta_D+c})=0$. Also, for $D=1$, $c=0$ and $d\ne 1$, $G_2(\eta_d, \chi_{\delta_D+c})=0$. In particular, we obtain the following simplified expression:
\begin{align*}
\frac{N(f, m, \beta)}{\theta(q^n-1)\Theta(f)}=\frac{1}{q^m} &\left[q^n + \sum_{c\in \F_{q^m}}\right. a_c\int\limits_{d|q^n-1\atop{d\ne 1}} \int\limits_{D|f\atop {D\ne 1}}G_2(\eta_d, \chi_{\delta_D+c})+ \left.\sum_{c\in \F^*_{q^m}}a_c\int\limits_{d|q^n-1\atop{d\ne 1}} G_2(\eta_d, \chi_{c})\right],
\end{align*}
as desired. From Theorem~\ref{Gauss2}, the remaining Gauss sums above satisfy the inequality $|G_2(\eta_d, \chi_{\delta_D+c})|\le 2q^{n/2}$. Since $|a_c|=1$ for any $c\in \F_{q^m}$, taking estimates in the previous sum we obtain the following inequality:

\begin{multline*}
\frac{N(f, m, \beta)\cdot q^m}{\theta(q^n-1)\Theta(f)}\ge \\ q^n-2q^{n/2}\cdot [q^m(W(q^n-1)-1)(W(f)-1)+(q^m-1)\cdot (W(q^n-1)-1)],
\end{multline*}
hence 
\[
\frac{N(f, m, \beta)}{\theta(q^n-1)\Theta(f)}> q^{n-m}-2q^{n/2}W(q^n-1)W(f). 
\]
\qed\end{proof}
%
%
%
%
%
Additionally, with some persistent cases in mind, we introduce some sieving techniques, as presented in \cite{cohen}. We note though, that for our cause it suffices to apply sieving solely on the multiplicative part. To this end, write $N(f,m,\beta)_{l}$ as the number of $l$-free elements $w\in\F_{q^n}$ such that $w^2$ is $f$-free over $\F_q$ and $\Tr_{q^n/q^m}(w^2)=\beta$, where $l\mid q'$ and $q'$ is the square-free part of $q^n-1$. Similarly, we set $N_2(n)_l$ as the number of $l$-free $w\in \F_{q^n}$ such that $w^2$ is normal over $\F_q$. In particular,  $N(f,m,\beta)_{q'} = N(f,m,\beta)$ and $N_2(n)_{q'}=N_2(n)$.
\begin{proposition}[Sieving inequality]\label{propo:sieve1}
Let $m,f,\beta$ be as in Theorem~\ref{thm:main_ineq} and let $\{r_1,\ldots ,r_s\}$ be divisors of $q'$, such that $\gcd(r_i,r_j)=q_0$ for all $i\neq j$ and $\lcm(r_1,\ldots ,r_s)=q'$, then
\[
N(f,m,\beta)_{l} \geq \sum_{i=1}^s N(f,m,\beta)_{r_i} -(r-1)N(f,m,\beta)_{q_0}
\]

and 

\[
N_2(n)_l\geq \sum_{i=1}^s N_2(n)_{r_i} -(r-1)N_2(n)_{q_0}.
\]
\end{proposition}
\begin{proof}
We just prove the first inequality, since the second follows in a similar way. Let $S(l)$ be the set of $l$-free elements of $\F_{q^n}$ whose squares have trace $\beta$ over $\F_{q^m}$ and are $f$-free over $\F_q$, where $l\mid q^n-1$. We will use induction on $s$. The result is trivial for $s=1$. For $s=2$, we have that $S(r_1)\cup S(r_2) \subseteq S(q_0)$ and $S(r_1)\cap S(r_2) = S(q')$. The result follows after considering the cardinalities of the above sets.

Next, suppose the desired result holds for some $s\geq 2$. For $s+1$, if we denote by $r_0$ the least common multiplier of $r_2,\ldots,r_{s+1}$, we observe that $r_0,r_1$ satisfy the conditions for $s=2$. The desired result follows from the induction hypothesis.
\qed\end{proof}
\begin{corollary}\label{cor:main_ineq}
Let $\{p_1 ,\ldots ,p_s \}$ be a set of distinct prime divisors of $q'$, the square-free part of $q^n-1$ (this set may be $\emptyset$, in which case $s=0$), such that $\delta := 1-\sum_{i=1}^s p_1^{-1} > 0$. Also, let $n=p^k u$, where $\gcd(u, p)=1$. The following hold.
\begin{enumerate}[(a)]
\item If
\begin{equation}\label{eq1}q^{n/2}>2\cdot W(q_0)W(x^u-1) \left( \frac{s-1}{\delta} + 2 \right) ,\end{equation}
there exists a normal element in $\F_{q^n}$ that is $2$-primitive.
\item If \begin{equation}\label{main_eq2}q^{p^k(u/2-1)}>W(q_0)W(x^u-1) \left( \frac{s-1}{\delta} + 2 \right) ,\end{equation}
there exists an $1$-normal element in $\F_{q^n}$ that is $2$-primitive.
\end{enumerate}
\end{corollary}
\begin{proof}
We begin with the second item. Take $\beta$ and $T$ as in Proposition~\ref{LR}. It then follows from Proposition~\ref{char1} that it suffices to show that $N(T,p^k,\beta)>0$.
Proposition~\ref{propo:sieve1} implies
\[
N(T,p^k,\beta) \geq \sum_{i=1}^s N(T,p^k,\beta)_{q_0p_i} - (s-1) N(T,p^k,\beta)_{q_0}.
\]
Next, we follow the same steps as in the proof of Theorem~\ref{thm:main_ineq} and get:
\[
\frac{N(T,p^k,\beta)}{\theta(q_0)} \geq \delta q^{n-p^k}-q^{n/2}W(q_0)W(x^u-1) \left( 1 + \sum_{i=1}^s \left( \theta(p_i) \frac{W(q_0p_i)}{W(q_0)} - 1 \right) \right) .
\]
The above, combined with the fact that $W(q_0p_i)/W(q_0)=2$, yields
\[
\frac{N(T,p^k,\beta)}{\theta(q_0)} \geq \delta q^{p^k(u-1)}-q^{n/2}W(q_0)W(x^u-1) ( s-1+2\delta ) ,
\]
which implies the desired result. The proof of the first item is almost identical and ommited.
\qed\end{proof}
The following corollary is also useful.

\begin{corollary}\label{cor:2-trace}
Let $n=p^k u$, where $\gcd(u, p)=1$ and $p\ge 5$. If $k\ge 2$, then for any $\beta\in \F_{q^{n/p}}$, there exists a $2$-primitive element $\alpha\in \F_{q^n}$ such that $\Tr_{q^n/q^{n/p}}(\alpha)=\beta$. 
\end{corollary}

\begin{proof}
According to Theorem~\ref{thm:main_ineq}, is enough to prove that $N(1, n/p, \beta)>0$. From the same theorem, the last inequality holds if 
$$q^{n-n/p}>2q^{n/2}W(q^n-1).$$

According to Proposition~\ref{propo:altbounds_w}, $W(\frac{q^{n}-1}{q^u-1})\le 2q^{\frac{(p^k-1)u}{2+\log_2p}}<2q^{\frac{(p^k-1)u}{4.3}}$ for $p\ge 5$ and we have the trivial bound $W(q^u-1)\le 2q^{u/2}$. Therefore, 
$$W(q^n-1)\le 4q^{\frac{n+1.15u}{4.3}},$$
and so $2q^{n/2}W(q^n-1)\le 8q^{\frac{3.15n+1.15u}{4.3}}$. Since  $k\ge 2$ and $p\ge 5$, 
$$n-\frac{n}{p}-\frac{3.15n+1.15u}{4.3}\ge \frac{1.15(p^k-1)u}{4.3}-\frac{n}{5}=\frac{(0.29 p^k-1.15)u}{4.3}.$$
The expression $\frac{(0.29 p^k-1.15)u}{4.3}$ achieves its minimum at $u=1, p=5$ and $k=2$, with value $\frac{0.29\cdot 25-1.15}{4.3}>1.4$. In particular, $q^{n-n/p-\frac{3.15n+1.15u}{4}}\ge q^{1.4}\ge 5^{1.4}>8$ and so we obtain the desired result.
\qed\end{proof}

\section{Existence results}

In this section, we use the theory developed in the previous ones in order to complete our results. Our procedure, in general, relies on verifying the inequalities of Corollary~\ref{cor:main_ineq} to obtain the existence of $2$-primitive, $k$-normal elements in $\F_{q^n}$ for $k=0, 1$. For the mentioned computations, the \textsc{SageMath} software was used.

First, notice that the case $n=2$ is elementary and the following lemma summarizes the results.

\begin{lemma}\label{lemma:n=2}
If $q>3$ is an odd prime power, then all $2$-primitive $c\in\F_{q^2}$ are normal over $\F_q$. In contrast, all $2$-primitive elements of $\F_{3^2}$ are $1$-normal over $\F_3$.
\end{lemma}
\begin{proof}
Observe that any nonzero element of $\F_{q^2}$ is either normal or $1$-normal. If $\alpha\in \F_{q^2}^*$ is $1$-normal, then $(x-b)\circ \alpha=0$ for some $b\in \F_q^*$, hence $\alpha^{q}=b\alpha$ and so $\alpha^{(q-1)^2}=1$. In particular, any element in $\F_{q^2}$ with multiplicative order greater than $(q-1)^2$ is normal. For $q>3$, $(q^2-1)/2>(q-1)^2$ and this implies the normality over $\F_q$ of all $2$-primitive elements. If $c\in\F_{3^2}$ is $2$-primitive, then it has multiplicative order $4$, and so $c^q=c^3=\pm c$. In other words, $c$ is $1$-normal.
\qed\end{proof}

So, from now on we assume that $n\geq 3$. Before we move to the computational part, we note that exact calculations or estimations of $W(t)$ will be needed. A simple combinatorial argument yields
\begin{equation}\label{eq:w1}
W(t) = 2^{d(t)} ,
\end{equation}
where $d(t)$ stands for the number of distinct prime numbers or monic irreducible polynomials dividing $t$, where $t$ can be a positive integer or a monic polynomial over $\F_q$, respectively. Additionally, the following bounds hold.

\begin{lemma}\label{lemma:w_bound1}
Let $t,a$ be positive integers, then $W(t) \leq c_{t,a} t^{1/a}$, where $c_{t,a} = 2^s/(p_1\cdots p_s)^{1/a}$ and $p_1,\ldots ,p_s$ are the primes $\leq 2^a$ dividing $t$. In particular we are interested in $c_t:=c_{t,4}$ and $d_t:=c_{t,8}$, and for every $t$ we have that $c_t<4.9$ and $d_t<4514.7$.
\end{lemma}
\begin{proof}
The statement is an immediate generalization of Lemma~3.3 of \cite{cohen} and can be proved using multiplicativity. The bounds for $c_t$ and $d_t$ can be easily computed.
\qed\end{proof}
\begin{remark}
From the above, $c_{t,a}$ depends on the primes $<2^a$ that divide $t$. It takes its maximum value when all primes $<2^a$ divide $t$ and this is how the above bounds were computed. Nonetheless, since it is uncommon for all such primes to actually divide $t$, it turns out to be quite smaller for most $t$ and, in any case, easily computable. In addition, if we know in advance that some prime does not divide $t$, then sharper estimates for $c_{t,a}$ can be used. For instance if $q=5^t$, then $5 \nmid q^n-1$, hence $d_{q^n-1}=c_{q^n-1,8}<2760.39$.
\end{remark}



\begin{lemma}\label{lemma:w_bound2}
We have that $W(x^u-1) \leq 2^{(u+\gcd(u,q-1))/2}$. In particular,
\begin{enumerate}
\item for every $q$, $W(x^u-1) \leq 2^{(u+\mathrm{min}(u,q-1))/2}$,
\item for $q=5$, $W(x^u-1) \leq 2^{(u/3)+6}$ and
\item for $q=3$, $W(x^u-1) \leq 2^{(u+1)/3}$, for $u\neq 4,6,8$.
\end{enumerate}
\end{lemma}
\begin{proof}
For the initial statement see Eq.~(2.10) of \cite{lenstra}. The first item is a direct consequence of this and for the other two see Lemma~2.11 in \cite{lenstra}.
\qed\end{proof}


\subsection{Existence of normal, $2$-primitive elements}

In this stage we are ready to investigate the existence of normal elements of $\F_{q^n}$ over $\F_q$ with multiplicative order $(q^n-1)/2$.
First, we combine the first item of Corollary~\ref{cor:main_ineq} with Eq.~\eqref{eq:w1}, Lemma~\ref{lemma:w_bound1} and the various estimates of Lemma~\ref{lemma:w_bound2} and get several (similar) conditions for the existence of the elements of our interest.
%
%
By using these conditions, we compile Table~\ref{tab:pan1}.
\begin{table}[h]
\begin{center}\scriptsize
\begin{tabular}{c||c|c|c|c|c|c|c|c|c|c|c}
$q$ & $\ge 336$ & $\ge 157$ & $\ge 100$ & $\ge 74$ & $\ge 51$ & $\ge 25$ & $\ge 17$ & $\ge 9$ & $=7$ & $=5$ & $=3$ \\ \hline
$n$ & $\ge 3$ & $\ge 4$ & $\ge 5$ & $\ge 6$ & $\ge 8$ & $\ge 18$ & $\ge 21$ & $\ge 25$ & $\ge 29$ & $\ge 35$ & $\ge 53$ 
\end{tabular}
\end{center}
\caption{Values for $q$ and $n$ such that Eq.~\eqref{eq1} holds
for $k=0$.
\label{tab:pan1}}
\end{table}
These results cover all cases, with the exception of 468 pairs $(q,n)$. 
Nonetheless, this list is shortened to a total of 165 pairs $(q,n)$ after $c_{q^n-1}$ is explicitly computed and, among those, only 34 pairs fail after all quantities are replaced by their exact values.
%
%
For those pairs we employ Corollary~\ref{cor:main_ineq} with $s>0$ and apply a sieving algorithm, based on the one presented in Appendix~\ref{appendix:pseudocodes}. This enables us to exclude additional 12 pairs. So we are left with the 22 pairs 
$(3, 3)$, $(5, 3)$, $(7, 3)$, $(11, 3)$, $(13, 3)$, $(19, 3)$, $(25, 3)$, $(31, 3)$,
$(3, 4)$, $(5, 4)$, $(7, 4)$, $(9, 4)$, $(11, 4)$, $(13, 4)$, $(17, 4)$, $(3, 5)$, $(11,
5)$, $(3, 6)$, $(5, 6)$, $(7, 6)$, $(3, 8)$ and $(5, 8)$.
%
%

We define $S_0$ as the set of the 22 pairs above. In particular, if $(q, n)\not \in S_0$, it follows that there exist $2$-primitive, normal elements of $\F_{q^n}$ over $\F_q$. For the pairs $(q, n)$ in $S_0$, we verify directly the existence of $2$-primitive, normal elements, for all pairs in $S_0$ with the sole exception of $(3,4)$, which, in fact, is a genuine exception. The \textsc{SageMath} program used for this verification is described in Appendix~\ref{appendix:pseudocodes}. In conclusion, we obtain the following result.
%
%

\begin{theorem}\label{2-p-normal}
Let $q$ be a power of a prime $p\ge 3$ and let $n\ge 3$ be a positive integer. Then there exist $2$-primitive, normal elements of $\F_{q^n}$ over $\F_q$, unless $(q,n)=(3,4)$.
\end{theorem}

\subsection{Existence of $1$-normal, $2$-primitive elements}
Next, we proceed to the $1$-normal and $2$-primitive case. We separate the study of Eq.~\eqref{main_eq2} in two cases whether $n$ is divisible by $p$ or not, i.e., $k=0$ or $k>0$, while for the former we study the case $n=3$ separately, since this case requires special attention and a different strategy is adopted. Also, we confine ourselves to $p\geq 5$.
\subsubsection{The case $p\neq 3$ and $n=3$}\label{section:n=3}
%

Here we assume that $n=3$ and $p\ge 5$. For this case, it is clear that a $2$-primitive element of $\F_{q^3}$ can be either normal or $1$-normal over $\F_q$; for this, we see that any $2$-normal element $\alpha\in \F_{q^3}$ is such that $\alpha^{q-1}$ is in $\F_q$, hence $\alpha^{(q-1)^2}=1$ but $(q-1)^2<(q^3-1)/2$ for $q\ge 5$. Additionally, some $w\in\F_{q^3}$ such that $\Tr_{q^3/q}(w)=0$ cannot be normal, that is it suffices to search for a primitive element $w\in\F_{q^3}$ such that $\Tr_{q^3/q}(w^2)=0$, so Theorem~\ref{thm:main_ineq} implies that it is enough to show that $N(1,1,0)>0$. A treatment similar to Corollary~\ref{cor:main_ineq} yields the following.
%
%
%
\begin{corollary}\label{cor:sieve2}
Let $\{p_1 ,\ldots ,p_s \}$ be a set of distinct prime divisors of $q'$, the square-free part of $q^3-1$ (this set may be $\emptyset$, in which case $s=0$), such that $\delta := 1-\sum_{i=1}^s p_1^{-1} > 0$. If
\[
q^{1/2} > 2 W(q_0) \left( \frac{s-1}{\delta} + 2 \right),
\]
where $q_0 := q'/(p_1 \cdots p_s)$, then there exists a 2-primitive $1$-normal element of $\F_{q^3}$ over $\F_q$.
\end{corollary}
%
We begin without any sieving primes, i.e. $s=0$, so a sufficient condition would be
\begin{equation}\label{eq:n=3(1)}
q^{1/2} > 2\cdot W(q^3-1).
\end{equation}
With the help of Lemma~\ref{lemma:w_bound1}, this gives $q > (2\cdot d_{q^3-1})^8$, where $d_{q^3-1}<4514.7$. This is true for $q\geq q_{\mathrm{max}} = 9029.4^8 \sim 4.42\cdot 10^{31}$. Recall that $d(t)$ denoted the number of distinct prime divisors of $t$. A quick computation shows that the product of any $52$ distinct primes is larger than $q_{\mathrm{max}}^3-1$, which implies that the case $d(q')\ge 52$ is settled.

Next, we focus on $q$ with $19\leq d(q')\leq 51$. We sieve the largest 15 prime divisors, say $p_1, \ldots ,p_{15}$, of $q'$ and we get that $\delta \geq 0.32$ and $d_{q'/p_1\cdots p_{15}} < 2618.07$. The resulting condition is satisfied for $q\geq 8261356$. However, for all $q< 8261356$, we have that $d(q') < 19$, i.e. this possibility is also settled.

We employ the same technique for $11 \leq d(q') <19$ and resolve the case $q\geq 906561$ for that case. For the remaing cases, it is clear that Eq.~\eqref{eq:n=3(1)} remains unsatisfied for $q<2^{22}$. So, since $2^{22}>906561$, it suffices to investigate $q<2^{22}$.

First, we check that all, but 42304 $q$'s satisfy Eq.~\eqref{eq:n=3(1)}, with the bound $W(q^3-1) \leq c_{q^3-1,12} q^{1/4}$. Amongst them, 10067 do not satisfy this with $W(q^3-1)$ being excplicitly computed. For those $q$'s we apply a sieving algorithm, as presented in Appendix~\ref{appendix:pseudocodes}, that is succesfull for all but 191 values of $q$. These prime powers are presented in Table~\ref{table:1n2p_exc1} and for them, we are able to find examples of $2$-primitive $1$-normal elements, using the algorithm presented in Appendix~\ref{appendix:pseudocodes}. All in all, we have the following.
\begin{proposition}
If $p> 3$, there exists a $2$-primitive $1$-normal element of $\F_{q^3}$ over $\F_q$. 
\end{proposition}
%
%
%
\subsubsection{The case $\gcd(n, p)=1$, $n,p>3$}
If $n$ is not divisible by $p$, then $n=u$ and Eq.~\eqref{main_eq2} is equivalent to 
\begin{equation}\label{eq:(n,p)=1}q^{n/2-1}>W(q^n-1)W(x^n-1).\end{equation}
The cases $n\leq 3$ have already been settled, so from now on we may assume that $n\geq 4$. For $n=4$, we check that for $q\geq 5217924120$, we have that
\begin{equation}\label{eq:1n2p_generic1}
q^{3n/8-1} \geq d_{q^n-1} \cdot W(x^u-1),
\end{equation}
for $d_{q^n-1} \leq 4514.7$ and $W(x^u-1) \leq 2^n$. Within the range $3< q< 5217924120$ we check that there are 4598 prime powers co-prime to $6$, that do not satisfy Eq.~\eqref{eq:1n2p_generic1}, if we substitute $d_{q^n-1}$ by its exact value, with $q=1658623$ being the largest amongst them. We check them and verify that 433 prime powers $q$ fail to satisfy
Eq.~\eqref{eq:(n,p)=1} with all quantities explicitly computed. Then, we follow the same steps for $n=5,\ldots,10$. 
%
%

For $11\leq n <20$, Eq.~\eqref{eq:1n2p_generic1} holds for $q\geq 170$, when $d_{q^n-1} \leq 4514.7$ and $W(x^u-1) \leq 2^n$. In this region, 41 pairs do not satisfy Eq.~\eqref{eq:1n2p_generic1}, once we replace $d_{q^n-1}$ by its exact value and among them only 3 fail to satisfy Eq.~\eqref{eq:(n,p)=1} with all quantities explicitly computed. For $21\leq n<25$, we follow the same procedure and initially we have get the desired result for $q\geq 31$, then, with the same steps, we reduce the list to 18 pairs, but eventually we have no new possible exceptions. Finally, we note that Eq.~\eqref{eq:1n2p_generic1} holds for $n\geq 25$ and $q\geq 22$, so the cases $q<21$ and $n>26$ are left to investigate.
%
%

So we start with $q=19$. A quick computation reveals that Eq.~\eqref{eq:1n2p_generic1} holds for $q=19$ and $n\geq 23$, for $d_{q^n-1} < 3261.7$ (as $19\nmid q^n-1$) and $W(x^u-1) \leq 2^{(n+q-1)/2}$, so we have no additional exceptions for $q=19$. Similar arguments for $q=17, 13, 11, 7$ and $5$ yield no further exceptions. We only describe $q=5$ as an example: in this case, we assume $d_{q^n-1} < 2760.39$ and Eq.~\eqref{eq:(n,p)=1} is true for $n\geq 43$, while in the region $26\leq n<43$, Eq.~\eqref{eq:(n,p)=1} is satisfied.
%
%

From the above procedure we have identified 472 possible exception pairs $(q,n)$. The sieving algorithm, as presented in Appendix~\ref{appendix:pseudocodes}, yields a succesfull result for roughly half of them, in particular for those with large components, while the remaining 283 pairs are presented in Table~\ref{table:1n2p_exc1}.
%
%

We define $S_1$ as the set of the pairs in  Table~\ref{table:1n2p_exc1}. In particular, if $(q, n)\not \in S_1$, it follows from Theorem~\ref{thm:main_ineq} that there exist $2$-primitive, $1$-normal elements of $\F_{q^n}$ over $\F_q$. For the pairs $(q, n)$ in $S_1$, we verify directly the existence of $2$-primitive, $1$-normal elements: see Appendix~\ref{appendix:pseudocodes} for the pseudocode. In conclusion, we obtain the following result.
%
%

\begin{proposition}\label{propo:2-p-$1$-normal}
Let $q$ be a power of a prime $p\ge 3$ and let $n\ge 3$ be a positive integer such that $\gcd(n, p)=1$. Then there exist $2$-primitive, $1$-normal elements of $\F_{q^n}$ over $\F_q$.
\end{proposition}

%
\begin{table}[ht]
\begin{center}\tiny
\begin{tabular}{|c|p{0.8\linewidth}|c|}
\hline $n$ & $q$ & \# \\ \hline\hline
3 & 5, 7, 11, 13, 17, 19, 23, 25, 29, 31, 37, 41, 43, 47, 49, 53, 59, 61, 67, 71, 73, 79, 83, 89, 97, 101, 103, 107, 109, 113, 121, 125, 127, 131, 137, 139, 149, 151, 157, 163, 169, 179, 181, 191, 193, 197, 199, 211, 223, 227, 229, 233, 239, 241, 251, 263, 269, 271, 277, 281, 283, 289, 307, 311, 313, 331, 337, 343, 347, 349, 359, 361, 367, 373, 379, 397, 401, 409, 419, 421, 431, 439, 443, 457, 461, 463, 487, 491, 499, 521, 523, 529, 541, 547, 571, 601, 607, 613, 619, 625, 631, 643, 661, 691, 709, 733, 739, 751, 757, 809, 811, 821, 823, 841, 859, 877, 907, 919, 961, 967, 991, 997, 1009, 1021, 1031, 1033, 1051, 1069, 1087, 1093, 1123, 1129, 1171, 1201, 1231, 1291, 1303, 1321, 1327, 1369, 1381, 1429, 1451, 1453, 1471, 1531, 1597, 1621, 1681, 1741, 1759, 1831, 1849, 1871, 1873, 1933, 1951, 2011, 2209, 2221, 2311, 2341, 2347, 2401, 2473, 2521, 2531, 2551, 2557, 2671, 2731, 2851, 2857, 2971, 3061, 3301, 3481, 3541, 3571, 3691, 3721, 3931, 4111, 4561, 4621, 4951, 5791, 5821, 6091, 9181, 9811 & 191 \\ \hline
4 & 5, 7, 11, 13, 17, 19, 23, 25, 29, 31, 37, 41, 43, 47, 49, 53,
59, 61, 67, 71, 73, 79, 83, 89, 97, 101, 103, 107, 109, 113, 121,
125, 127, 131, 137, 139, 149, 151, 157, 167, 169, 173, 179, 181, 191,
193, 197, 211, 229, 233, 239, 241, 257, 269, 277, 281, 293, 307, 313,
317, 337, 349, 353, 361, 373, 389, 397, 401, 421, 433, 461, 463, 701,
853 & 74 \\ \hline
5 & 7, 11, 31, 61 & 4 \\ \hline
6 & 5, 7, 11, 13, 19, 25, 31, 37, 43 & 9 \\ \hline
8 & 5, 7 & 2 \\ \hline
10 & 11 & 1 \\ \hline
12 & 5, 7 & 2 \\ \hline
\multicolumn{2}{|r|}{\textbf{Total:}} & 283 \\ \hline
\end{tabular}
\end{center}
\caption{Possible exceptions pairs $(q,n)$ for $2$-primitive and $1$-normal, when $p\geq 5$, $n\geq 3$ and $\gcd(p,n)=1$.\label{table:1n2p_exc1}}
\end{table}

\subsubsection{The case $n$ is divisible by $p$}

Here, we confine ourselves to the case $p\neq 3$.
Following the ideas of \cite{RT}, we first proceed with the case when $n$ is divisible by $p^2$, i.e., $k\ge 2$. We obtain the following result.

\begin{lemma}
Suppose that $n=p^2s$, where $s\ge 1$ and $p\ge 5$. There exists a $2$-primitive, $1$-normal element of $\F_{q^n}$ over $\F_q$.  
\end{lemma}

\begin{proof}
Let $\beta\in \F_{q^{ps}}$ be such that $m_{\beta}=\frac{x^{ps}-1}{x-1}$. From Corollary~\ref{cor:2-trace}, there exists a $2$-primitive element $\alpha\in \F_{q^n}$ such that $\Tr_{q^n/q^{ps}}(\alpha)=\beta$. From Proposition~\ref{prop:1-trace}, it follows that such an $\alpha$ satifies $m_{\alpha}=\frac{x^n-1}{x-1}$, i.e. $\alpha$ is $2$-primitive and $1$-normal.
\qed\end{proof}

We proceed with the case $u\geq 3$ and $k=1$.



\begin{proposition}\label{propo:u>2}
Let $q=p^t$ and $n=pu$, where $p>3$, $\gcd(p,u)=1$, $u>2$ and $k\geq 1$. Then there exists an $1$-normal and $2$-primitive element of $\F_{q^n}$ over $\F_q$.
\end{proposition}
\begin{proof}
We start with the delicate case $u=3$. Write $q=p^t$. For $p=5$, Lemma~\ref{lemma:w_bound1} implies $W(q^n-1)<d_{q^n-1}\cdot q^{n/8}$ and Lemma~\ref{lemma:w_bound2} $W(x^u-1)\leq 2^3$. In particular, we use the bound $d_{q^n-1}<2760.39$. We attach these bounds to Eq.~\eqref{main_eq2} and get that we have our desired result for $t\geq 10$. We use the same technique for $7\leq p\leq 17$ and settle the cases for $t$ larger than a specific number 
and for $p=17$ (with $d_{q^n-1}<3216.66$) we settle the case $t\geq 2$. 
Finally, for $p\geq 19$ and $t=1$, we use the generic bound $d_{q^n-1}<4514.7$ and Eq.~\eqref{main_eq2} holds for $q\geq 29$. In short, there remain 22 pairs $(q,n)$ to deal with.
%
%

We continue with the remaining cases. We observe that Eq.~\eqref{main_eq2} is always satisfied (with the generic bounds $W(q^n-1)<4.9\cdot q^{n/4}$ and $W(x^u-1)\leq 2^u$) for $u\geq 8$. 
With similar techniques, we identify the possible exceptions for $4\leq u\leq 7$ and we end up with additional 30 pairs $(q,n)$ to deal with.
%
%

The combined list of 52 pairs $(q,n)$ of possible exceptions is reduced to 7, once we use the bound $W(q^n-1)\leq d_{q^n-1}q^{n/8}$ in Eq.~\eqref{main_eq2}, but compute $d_{q^n-1}$ and $W(x^u-1)$ explicitly for every pair. Amongst those we find the persistent pairs $(5,15)$, $(5,20)$ and $(25,15)$ that do not satisfy Eq.~\eqref{main_eq2}, with all quantities explicitly computed,
%
%
but we succesfully apply sieving on them.
%
%
\qed\end{proof}

We now add the cases $n=p, 2p$. We note that Eq.~\eqref{main_eq2} is not useful in these cases. In the following proposition, we combine Eq.~\eqref{eqn:main1} with some combinatorial arguments in order to obtain the desired existence result. 

\begin{proposition}
Let $p\ge 5$ be a prime and $n=p, 2p$.  If there is no $2$-primitive, $1$-normal element of $\F_{q^n}$ over $\F_q$, then

\begin{enumerate}
\item $n=p$ and 
\begin{equation}\label{eq0:n=p}\frac{1}{\theta(q^{p}-1)}>q-2q^{2-\frac{p}{2}}W(q^{p}-1),\end{equation}
which is violated, except for $q=p=5$.
\item $n=2p$ and
\begin{equation}\label{eq0:n=2p}\frac{q}{(q-1)\theta(q^{2p}-1)}>q-4q^{2-p}W(q^{2p}-1),\end{equation}
which is violated for any $q$.

\end{enumerate}
\end{proposition}
\begin{proof}
We split the proof into cases.

\begin{itemize}
\item Case $n=p$. Let $N$ be the number of $2$-primitive, $1$-normal elements of $\F_{q^n}$. We observe that if $\alpha\in \F_{q^n}$ is $2$-primitive and has $0$-trace over $\F_q$, then $m_{\alpha}=(x-1)^d$ for some $0\le d\le p-1$. With the notation of Theorem~\ref{thm:main_ineq}, we know that the number of $2$-primitive elements of $0$-trace equals $N(1, 1, 0)$. Clearly, if $N=0$, then any $2$-primitive element $\alpha$ of $0$-trace satisfies $(x-1)^{p-2}\circ \alpha=0$: this equation has at most $q^{p-2}$ solutions. This shows that, if there is no $2$-primitive, $1$-normal element in $\F_{q^n}$, $N(1, 1, 0)\le q^{p-2}$. However, according to Eq.~\eqref{eqn:main1},
 $$\frac{N(1, 1, 0)}{\theta(q^{p}-1)}> q^{p-1}-2q^{p/2}W(q^{p}-1),$$
hence 
$$\frac{q^{p-2}}{\theta(q^{p}-1)}> q^{p-1}-2q^{p/2}W(q^{p}-1),$$
and so
$$\frac{1}{\theta(q^{p}-1)}>q-2q^{2-\frac{p}{2}}W(q^{p}-1).$$

According to Lemma~\ref{lem:ramanujan}, $\frac{1}{\theta(q^{p}-1)}\le 3.6\log q$. The later implies the following inequality:

\begin{equation}\label{eq:n=p}3.6\log q>q-2q^{2-\frac{p}{2}}W(q^{p}-1).\end{equation}

We first suppose that $q\ne 5$. According to Proposition~\ref{propo:altbounds_w}, $W(\frac{q^{p}-1}{q-1})\le q^{\frac{p-1}{4.3}}$. We have the trivial bound $W(q-1)\le 2q^{1/2}$ and so $W(q^p-1)\le 2q^{\frac{p+1.15}{4.3}}$. In particular, from Eq.~\eqref{eq:n=p}, it follows that

$$3.6\log q>q-4q^{\frac{9.75-1.15p}{4.3}}.$$

Write $q=p^t$. If $p=5$, the last inequality implies $3.6\log q>q-4q^{\frac{4}{4.3}}$, which is true only for $t\le 12$. For $p=7$, we get $3.6\log q>q-4q^{\frac{1.7}{4.3}}$, which is true only for $t=1$. For $p\ge 11$, we have $3.6\log q>q-4$, which is true only for $q\le 13$. 

For the remaining cases $q=5^t, t\le 12$, $q=7, 11$ and $13$, we go back to Eq.~\eqref{eq0:n=p} and, replacing $q$ and $p$ by their exact values, we see that with the exception of the case $q=5$, Eq.~\eqref{eq0:n=p} does not hold.

\item Case $n=2p$. Let $N$ be the number of $2$-primitive, $1$-normal elements of $\F_{q^n}$. We note that if $\alpha\in \F_{q^n}$ is $2$-primitive, $(x+1)$-free and has $0$-trace over $\F_q$, then $m_{\alpha}=(x+1)^p(x-1)^d$ for some $0\le d\le p-1$. With the notation of Theorem~\ref{thm:main_ineq}, we observe that the number of $2$-primitive elements of $0$-trace equals $N(x+1, 1, 0)$. If $N=0$, it follows that any $2$-primitive element that has trace $0$ over $\F_q$ and is $(x+1)$-free satisfies $(x+1)^p(x-1)^{p-2}\circ \alpha=0$; this equation has at most $q^{n-2}$ solutions. In particular, if there is no $2$-primitive, $1$-normal element in $\F_{q^n}$, $N(x+1, 1, 0)\le q^{p-2}$. However, according to Eq.~\eqref{eqn:main1},
 $$\frac{N(x+1, 1, 0)}{\theta(q^{2p}-1)\Theta(x+1)}> q^{2p-1}-2q^{p/2}W(q^{2p}-1)W(x+1),$$
hence 
$$\frac{q^{2p-1}}{(q-1)\theta(q^{2p}-1)}> q^{2p-1}-4q^{p}W(q^{2p}-1),$$
and so
$$\frac{q}{(q-1)\theta(q^{2p}-1)}>q-4q^{2-p}W(q^{2p}-1).$$
According to Lemma~\ref{lem:ramanujan}, $\frac{1}{\theta(q^{2p}-1)}\le 3.6\log q+1.8\log 2<3.6\log q+1.25$; this implies the following inequality
\[\frac{q}{q-1}(3.6\log q+1.25)>q-2q^{2-p}W(q^{2p}-1).\]

We first  suppose that $q\ne 5$. From Proposition~\ref{propo:altbounds_w}, $W(\frac{q^{2p}-1}{q^2-1})\le q^{\frac{2(p-1)}{4.3}}$. We have the trivial bound $W(q^2-1)\le 2q$ and so $W(q^{2p}-1)\le 2q^{\frac{2p+2.3}{4.3}}$. In particular, from Eq.~\eqref{eq:n=p}, it follows that

$$\frac{q}{q-1}(3.6\log q+1.25)>q-8q^{\frac{10.9-2.3p}{4.3}}\ge q-8q^{\frac{-0.6}{4.3}}.$$

A simple verification shows that this last inequality holds only for $q\le 17$. For the cases $q=5, 7, 11, 13, 17$, we go back to Eq.~\eqref{eq0:n=2p} and, replacing $q$ and $p$ by their exact values, we see that Eq.~\eqref{eq0:n=2p} does not hold for any of these values of $q$.
\qed
\end{itemize}
\end{proof}
As a final step, we find an example of $2$-primitive and $1$-normal element of $\F_{5^5}$ over $\F_5$ with Algorithm~\ref{alg:search} and this concludes the proof of Theorem~\ref{thm:main_result}.

%
%
%
\begin{remark}
In the present text, the case of $2$-primitive, $1$-normal elements of $\fqn$ over $\F_q$ is absent for fields of characteristic $p=3$. However, partial existence results for this case are feasible with our tools, but the absence of this case in Conjecture~\ref{MA} suggest that a pursuit for a complete result may be unrealistic.
\end{remark}
\section{A note on $2$-primitive elements with prescribed trace}
Before concluding, we make a small note regarding the existence of $2$-primitive elements with prescribes trace. First observe that, although for our purposes we expected the trace of our elements to take specific values, the conditions that were used, that led to the compilation of Table~\ref{table:1n2p_exc1}, are identical for every possible trace.

This implies that for every odd prime power $q$ and positive integer $n\geq 3$, such that $\gcd(q,n)=1$ and $(q,n)$ not present in Table~\ref{table:1n2p_exc1}, and for every $\alpha\in\F_q$, there exists some $2$-primitive $\beta\in\fqn$ such that $\Tr_{q^n/q}(\beta)=\alpha$. In addition to the pairs present in Table~\ref{table:1n2p_exc1}, we also consider 11 pairs $(q,n)$, with $q$ a power of $3$, that would appear in the table, if the powers of $3$ were not explicitly excluded.

For those 294 pairs, we first check the condition
\[
q^{n/2-1} > 2 W(q_0) \left( \frac{s-1}{\delta} + 2 \right),
\]
where $q_0$ and $s$ as in Corollary~\ref{cor:main_ineq}.
That settles some pairs and for the remaining 208 pairs we
utilize Algorithm~\ref{alg:prescribed_trace} from Appendix~\ref{appendix:pseudocodes} to find  $2$-primitive elements of $\fqn$ with trace equal to every element of $\F_q$.

Likewise, the only possible exception for the case $\gcd(q,n)>1$ and $p\geq 5$ would be the pair $(q,n)=(5,5)$, which we also verify with Algorithm~\ref{alg:prescribed_trace}. So, summing up, we have the following.
%
%
%
\begin{theorem}\label{thm:prescribed_trace}
Let $q$ be a power of an odd prime number $p$ and $n\geq 3$ such that either $\gcd(p,n)=1$ or $p\geq 5$, then for every $\alpha\in\F_q$ there exists some $2$-primitive element $\beta\in\fqn$ such that $\Tr_{q^n/q}(\beta)=\alpha$.
\end{theorem}

\begin{acknowledgements}
The first author was supported by T\"UB\.{I}TAK Project Number 114F432 and the
second author was supported by the Program CAPES-PDSE (process - 88881.134747/2016-01) at Carleton University. 
\end{acknowledgements}

\appendix

\section{
Estimates for $W(q^n-1)$ and $\theta(q^n-1)$}\label{sec:estimations}
Here we present some estimates to the numbers $W(q^{n}-1)$ and $\theta(q^n-1)$ with $n=p^k\cdot u$, $k\ge 1$: the estimates are based in elementary results in number theory and follows the same ideas contained in Appendix A of \cite{RT}. For this reason, we skip some details and just apply the results of \cite{RT}.

\begin{proposition}\label{propo:altbounds_w}
Let $q$ be a power of an odd prime $p$ and $n=p^k\cdot u$, where $k\ge 1$ and $\gcd(u, p)=1$. If $p\ge 5$, the following holds $$W\left(\frac{q^{p^ku}-1}{q^u-1}\right)\le \varepsilon(q)\cdot q^{\frac{(p^k-1)u}{2+\log_2 p}},$$
where $\varepsilon(q)=2$ if $q=5$ and $\varepsilon(q)=1$ if $q\ne 5$.
\end{proposition}

\begin{proof}
For $1\le i\le k$, set $\ell_i=\frac{q_{i-1}^p-1}{q_{i-1}-1}$, where $q_i=q^{p^iu}$. Notice that $\frac{q^{p^ku}-1}{q^u-1}=\prod_{i=1}^{k}\ell_i$, hence $W\left(\frac{q^{p^ku}-1}{q^u-1}\right)\le \prod_{i=1}^kW(\ell_i)$. According to item (i) of Proposition A.8 in \cite{RT}, $$W(\ell_i)\le q_{i-1}^{\frac{(p-1)}{2+\log_2p}}$$ with the exception of $q_i=5$. Notice that $q_i=5$ only if $i=0$, $q=5$ and $u=1$: in this case, $W(\ell_1)=W\left(\frac{5^5-1}{5-1}\right)=8<2\cdot 5^{4/(2+\log_25)}=2q_0^{\frac{(p-1)}{2+\log_2p}}$. This shows that, for any $p\ge 5$, 
\[\prod_{i=1}^kW(\ell_i)\le \varepsilon(q) \cdot \left(\prod_{i=0}^{k-1}q_i\right)^{\frac{(p-1)}{2+\log_2p}}=\varepsilon(q)\cdot q^{\frac{(p^k-1)u}{2+\log_2p}}. \]
\qed\end{proof}

\begin{lemma}\label{lem:ramanujan}
Suppose that $q$ is a power of a prime $p$, where $p\ge 5$. For $u\ge 1$, we have
\begin{equation*}\label{eqn:ramanujan}\theta(q^{pu}-1)^{-1}=\frac{q^{pu}-1}{\varphi(q^{pu}-1)}< 3.6 \log q+1.8\log u.\end{equation*}
\end{lemma}

\begin{proof}
It is well known that $\frac{n}{\varphi(n)}\le e^{\gamma}\log\log n+\frac{3}{\log\log n}$, for all $n\ge 3$, where $\gamma$ is the Euler constant and $1.7<e^{\gamma}<1.8$. Also, since $e^x\ge 1+x$ for any $x\ge 0$, we have $\log \log q\le \log q -1$. Therefore, $$\log\log (q^{pu}-1)<\log p+\log u+\log\log q \le \log u+2\log q-1.$$

By the hypothesis, $\log \log (q^{pu}-1)\ge \log\log (5^5-1)>2$, hence $\frac{3}{\log\log (q^{pu}-1)}\le 1.5$ and so we get the following:
\[\theta(q^{pu}-1)^{-1}\le 1.8(2\log q+\log u-1)+1.5< 3.6 \log q+1.8\log u. \]
\qed\end{proof}

\section{
Pseudocode for search for primitive $1$-normals}\label{appendix:pseudocodes}

In this section we explain the main algorithms used to verify our results in this paper.
In Algorithm~\ref{alg:sieve}, we present the sieving algorithm pseudocode we used to exclude pairs $(q,n)$ from the list of possible exceptions for the existence of $2$-primitive $1$-normal elements of $\fqn$ over $\F_q$. The other sieving algorithms mentioned in the text follow the same pattern.

\begin{algorithm}[h!]
\caption{Pseudocode for sieving algorithm}
\label{alg:sieve}
\begin{algorithmic}
\State{{\bf Input:} positive integers $q,n$}
\State {\bf Returns:}
\begin{itemize}
\item (``Success'', $\{ p_m, \ldots , p_{m-t} \}$) if sieving works with the mentioned primes;
\item (``Fail'', ``$\delta\leq 0$'') if during the procedure we get $\delta\leq 0$;
\item (``Fail'', ``No more primes'') if there are no more primes to sieve
\end{itemize}
\State{{\bf calculate:} $u, p^k, W(x^u-1)$ from $q,n$}
\State{$\mathrm{divs}= \{ p_1 ,\ldots ,p_m \}$} \Comment{The prime divisors of $q^n-1$ in increasing order}
\State{$\delta \gets 1$}
\State{$\mathrm{sieving\_primes} \gets \{ \}$}

\While{$\delta > 0$}
   \State{$t \gets \#\mathrm{sieving\_primes}$}
   \State{$\Delta \gets (k-1)/\delta+2$}
   \If{$q^{p^k(u/2-1)} > 2^{m-k}W(x^u-1)\Delta$}
	   \State \Return{(``Success'', $\mathrm{sieving\_primes}$)}
	\EndIf 
	
   \If{$k$ {\bf is} $m$}
	  \State \Return{(``Fail'', ``No more primes'')}
   \EndIf
   \State{$\delta \gets \delta - 1/p_{m-k}$}
   \State{$\mathrm{sieving\_primes} \gets \mathrm{sieving\_primes} \cup \{ p_{m-k} \}$}
\EndWhile
\State \Return{(``Fail'', ``$\delta \leq 0$'')}
\end{algorithmic}
\end{algorithm}

Following the approach of \cite{RT}, Algorithm~\ref{alg:search} presents a search routine for $2$-primitive, $1$-normal elements of $\fqn$ over $\fq$. This search is based on the original characterization of $k$-normal elements from~\cite{HMPT}. 

\begin{theorem}\label{thm:k-normalgcd}
Let $\alpha \in \fqn$ and let $g_\alpha(x) = \sum_{i=0}^{n-1} \alpha^{q^i}x^{n-1-i} \in \fqn[x]$. Then $\gcd(x^n-1, g_\alpha(x))$ has degree $k$ if and only if $\alpha$ is a $k$-normal element of $\fqn$ over $\fq$.
\end{theorem}

Algorithm~\ref{alg:search} proceeds as follows. Let $\fqn \cong \fq[x]/(f)$ with $f$ a primitive polynomial, and let $g$ be a root of $f$. Hence, $g$ is a generator of $\fqn^*$ and $g^i$ is $2$-primitive if and only if $\gcd(i, q^n-1) = 2$. 
For each $2$-primitive element, check its $k$-normality using Theorem~\ref{thm:k-normalgcd}. If $k=1$, the resulting element is $2$-primitive, $1$-normal and is returned. The algorithm returns ``Fail'' if no $2$-primitive $1$-normal is found after $q^n-2$ iterations; that is, if all of $\fqn^*$ is traversed. 

\begin{algorithm}[h!]
\caption{Pseudocode for $2$-primitive $1$-normal element search algorithm}
\label{alg:search}
\begin{algorithmic}
\State{{\bf Input:} positive integers $q,n$}
\State {\bf Returns:} $2$-primitive $2$-normal element: ${\mathrm{elt}} \in \fqn$; otherwise ``Fail''
\State{$\mathrm{mult\_order} \gets q^n-1$}
\State{$g \gets \mathrm{generator}(\fqn^*)$}
\State{${cyclo} \gets x^n-1 \in \fqn[x]$}
\State{}

\Function{$\mathrm{check\_k\_normal}$}{$v$} \Comment{See Theorem~\ref{thm:k-normalgcd}}
   \State $g_v(x) \gets v x^{n-1} + v^q x^{n-2} + \cdots + v^{q^{n-1}}$
	 \State $k \gets \deg(\gcd(g_v, {cyclo}))$
	 \State \Return $k$
\EndFunction

\State{}
\State{$i \gets 1$}
\While{True}
   \If{$i$ {\bf is} $\mathrm{mult\_order}$} \Comment{No $2-$primitive $1$-normals found in $\fqn$}
	   \State \Return{``Fail''}
	 \EndIf 
	
	 \If{$\gcd(i,\mathrm{mult\_order}) \neq 2$} \Comment{Only check $2$-primitive elements}
	    \State{$i \gets i+1$}
	    \State{{\bf continue}}
	 \EndIf
	 
	 \State{$\mathrm{elt} \gets g^i$}
	 \State{$k \gets \mathrm{check\_k\_normal}({elt})$}
	 \If{$k$ {\bf is} $1$}
	   \State \Return $\mathrm{elt}$
	 \EndIf
   \State{$i \gets i+1$}
\EndWhile
\end{algorithmic}
\end{algorithm}

Finally, Algorithm~\ref{alg:prescribed_trace} computes the traces of the $2$-primitive elements of $\fqn$ and stores the possible distinct values to a set. If, at any point, the set grows enough to reach cardinality $q$, this means that $\forall\alpha\in\F_q$, there exists some $2$-primitive $\beta\in\fqn$ such that $\Tr_{q^n/q}(\beta)=\alpha$. If not, then we have an exception. 

We implemented all algorithms with the \textsc{SageMath} computer algebra system.

\begin{algorithm}[h!]
\caption{Pseudocode for $2$-primitive elements with prescribed trace search algorithm}
\label{alg:prescribed_trace}
\begin{algorithmic}
\State{{\bf Input:} positive integers $q,n$}
\State {\bf Returns:} ``Success'' if such elements are found; otherwise ``Fail''
\State{$\mathrm{mult\_order} \gets q^n-1$}
\State{$g \gets \mathrm{generator}(\fqn^*)$}
\State{$\mathrm{traces} \gets \emptyset$}
\State{}
\State{$i \gets 1$}
\While{$i<\mathrm{mult\_order}$}
   \If{$\# \mathrm{traces}$ {\bf is} $q$}
	   \State \Return{``Success''}
	 \EndIf 
	
	 \If{$\gcd(i,\mathrm{mult\_order})$ {\bf is} $2$} \Comment{Only check $2$-primitive elements}
	   \State{$\mathrm{traces} \gets \mathrm{traces}\cup\{ \Tr_{q^n/q}(g^i) \}$}
	 \EndIf
   \State{$i \gets i+1$}
\EndWhile
\State \Return{``Fail''}
\end{algorithmic}
\end{algorithm}


\end{document}